\documentclass{amsart}
\usepackage{graphicx} 
\usepackage{amsmath,latexsym}
\usepackage{amsfonts,amssymb}
\usepackage{graphicx}
\usepackage[mathcal]{euscript}
\usepackage{amstext}
\usepackage{mathrsfs}
\usepackage[ansinew]{inputenc}
\usepackage{bm}
\usepackage{color}

\usepackage{enumitem}

\usepackage{mathtools}

\usepackage[nospace]{cite}

\usepackage[pagebackref=true, colorlinks=true, citecolor=blue]{hyperref}

\usepackage{upgreek}

\newtheorem{theo}{Theorem}[section]

\newtheorem{prop}{Proposition}[section]

\theoremstyle{definition}

\newtheorem{remark}{Remark}[section]
\newtheorem*{notation}{Notation}

\newcommand{\RR}{\mathbb{R}}

\newcommand{\bu}{\mathbf{u}}

\newcommand{\pa}{\partial}

\newcommand{\ds}{\displaystyle}
\newcommand{\dive}{\text{\normalfont div\,}}
\newcommand{\curl}{\text{\normalfont curl\,}}

\definecolor{Green}{rgb}{0.010,0.7,0.02}

\newcommand{\beq}{\begin{equation}}
\newcommand{\eeq}{\end{equation}}

\newcommand{\bal}{\begin{align}}
\newcommand{\eal}{\end{align}}

\newcommand{\cX}{\mathcal{X}}

\newcommand{\bv}{\mathbf{v}}

\newcommand{\be}{\bm{e}}

\newcommand{\ep}{\varepsilon}

\newcommand{\norm}[1]{\|#1\|}
\newcommand{\dt}[1]{\frac{d #1}{dt}}

\newcommand{\real}{\mathbb{R}}

\definecolor{magenta}{rgb}{1.0, 0.0, 1.0}
\definecolor{deepmagenta}{rgb}{0.8, 0.0, 0.8}
\definecolor{electricpurple}{rgb}{0.75, 0.0, 1.0}
\definecolor{lavenderindigo}{rgb}{0.58, 0.34, 0.92}
\definecolor{deeppink}{rgb}{1.0, 0.08, 0.58}
\definecolor{vividviolet}{rgb}{0.62, 0.0, 1.0}
\definecolor{darkorchid}{rgb}{0.6, 0.2, 0.8}
\definecolor{dartmouthgreen}{rgb}{0.05, 0.5, 0.06}
\definecolor{green(ncs)}{rgb}{0.0, 0.62, 0.42}
\definecolor{lawngreen}{rgb}{0.49, 0.99, 0.0}
\definecolor{green(pigment)}{rgb}{0.0, 0.65, 0.31}
\definecolor{green(html/cssgreen)}{rgb}{0.0, 0.5, 0.0}
\definecolor{neongreen}{rgb}{0.22, 0.88, 0.08}
\definecolor{orange(colorwheel)}{rgb}{1.0, 0.5, 0.0}
\definecolor{phlox}{rgb}{0.87, 0.0, 1.0}
\definecolor{tangelo}{rgb}{0.98, 0.3, 0.0}
\definecolor{yellow}{rgb}{1.0, 1.0, 0.0}
\definecolor{cadmiumyellow}{rgb}{1.0, 0.96, 0.0}
\definecolor{laserlemon}{rgb}{1.0, 1.0, 0.13}
\definecolor{uscgold}{rgb}{1.0, 0.8, 0.0}
\definecolor{uclagold}{rgb}{1.0, 0.7, 0.0}
\definecolor{darkspringgreen}{rgb}{0.09, 0.45, 0.27}
\definecolor{darkpastelgreen}{rgb}{0.01, 0.75, 0.24}
\definecolor{brightgreen}{rgb}{0.4, 1.0, 0.0}

\title[2D Ideal Boussinesq and density-dependent Euler]{Long-time existence for the 2D ideal Boussinesq and the 2D density-dependent Euler equations}

\author[H. Bae et Al.]{Hantaek Bae}
\author[]{Milton C. Lopes Filho}
\author[]{Anna L. Mazzucato}
\author[]{Helena J. Nussenzveig Lopes}

\date{\today}

\begin{document}

\begin{abstract} We establish long-time existence of smooth solutions to the 2D ideal Boussinesq equations and to the 2D non-homogeneous incompressible Euler equations for initial data consisting of small temperature perturbations, or small density perturbations, of smooth initial flows which are not necessarily small. Both results are known, see \cite{D13,DF11}, but the technique we develop to prove them is at the same time elementary and has broad potential applicability.
\end{abstract}

\maketitle

Dedicated to Peter Constantin, on the occasion of his 75$^{th}$ birthday.

\section{Introduction} \label{s:intro}

The aim of this work is to study existence of strong solutions on an arbitrary interval of time for the two-dimensional (2D for short) ideal Boussinesq equations and the 2D inviscid density-dependent, incompressible Euler equations on the whole plane, provided the initial temperature, respectively the initial density, are sufficiently close to a constant. The 2D Boussinesq equations are a model for convection in geophysical flows, while the 2D density-dependent Euler equations are a model for the motion of multi-phase, immiscible fluids.

Global existence holds for strong solutions of the (constant density) 2D incompressible Euler equations, a consequence of 2D vorticity transport. In fact, global existence is known even at the level of weak solutions with bounded vorticity. However, for the Boussinesq equations and for the density-dependent incompressible Euler equations, no global existence is known, because the vorticity equation is no longer a transport equation and vortex stretching is possible. In both cases, we have local well-posedness of smooth solutions, see \cite{CN97,CL03,BLS20} and long-time well-posedness also follows, if the initial data, both the velocity and the density or temperature perturbation, are small enough. The objective is to prove long-time well-posedness with no smallness assumption on the initial velocity, still taking advantage of vorticity transport.

In both problems, we identify a {\it continuation norm}, which, if finite, allows the smooth solution to be extended forward in time. We then derive an inequality for the continuation norm which uses the 2D structure of the problem, where we estimate the time derivative of the continuation norm in terms of a tame nonlinear part which, taken by itself, would allow for a long-time bound on the continuation norm, plus another term, which is history-dependent and nonlinear in a bad way, but which is {\it small} in terms of the size of the perturbation parameter.
Our main contribution is to derive a Gronwall-type estimate for the resulting history-dependent inequality, valid for a time which can be made as large as we want, by taking the perturbation parameter sufficiently small. The proof is elementary and adaptable to other similar contexts.

This article is divided in three sections. In the remainder of this section we introduce the Boussinesq and the non-homogeneous incompressible Euler equations, set notation and state our main results. In Section 2 we prove the simpler Boussinesq result and in Section 3 we prove the non-homogeneous Euler result, outlining those aspects of the proof that repeat the previous proof and add some concluding remarks.

Let us first write  the 2D Boussinesq equations on $\RR^2$ with zero viscosity and zero thermal diffusivity:
\begin{equation} \label{eq:Boussinesq}
  \begin{cases}
       \pa_t \bu + \bu \cdot \nabla \bu = - \nabla p + \theta \, \be_2, \\
    \pa_t \theta + \bu\cdot \nabla \theta =0, \\
    \dive \bu =0,
  \end{cases}
\end{equation}
where $\bu$ represents the fluid velocity, $p$ the fluid pressure, $\theta$ the temperature and $\be_2=(0,1)$. We choose the direction in which the buoyancy force acts to be fixed and agreeing with the vertical direction for simplicity. Our results hold for an arbitrary direction as long as it is fixed.

Similarly, we write the 2D non-homogeneous Euler equations on $\RR^2$ as the following system:
\begin{equation} \label{eq:nhEuler}
  \begin{cases}
       \rho \left(\pa_t \bv + \bv \cdot \nabla \bv\right) = - \nabla p , \\
    \pa_t \rho + \bv\cdot \nabla \rho =0, \\
    \dive \bv =0,
  \end{cases}
\end{equation}
where $\bv$ is  the fluid velocity, $p$ the fluid pressure, and $\rho$ the density.

\begin{notation}
In multiline formulas, the subscript on the label indicates the corresponding line, e.g. equation $\text{\eqref{eq:Boussinesq}}_2$ means the second equation of the system \eqref{eq:Boussinesq}. We also use the shorthand notation for functions of space and time $f(t)(x)=f(t,x)$, $t\geq 0$, $x\in \RR^2$, and write $f\in C_t(H^3(\RR^2))$ instead of $f\in C([0,T);H^3(\RR^2))$,  whenever the dependence on the time interval is clear. The space $H^s(\mathbb{R}^2)$ is the standard $L^2$-based Sobolev space. For notational ease, we may also drop the explicit dependence on the spatial domain $\RR^2$, which is fixed throughout. Lastly, $C$ denotes a generic constant that may vary from line to line and may depend on the norm of the initial data.
\end{notation}

We observe that, formally, the ideal Boussinesq equations reduce to the standard  Euler equations if $\theta_0$, and hence $\theta$, is a constant. Similarly, the density-dependent Euler equations become the standard Euler equations if $\rho_0$, and hence $\rho$, is a constant.

Short time existence and uniqueness of strong solution to \eqref{eq:Boussinesq} is  known, see  \cite[Theorem 2.1]{CN97} and \cite{D13,MP19}.  Short-time existence and uniqueness of strong solutions to  \eqref{eq:nhEuler} has been extensively investigated both in Sobolev spaces, as well as in Besov spaces, in bounded domains and in $\RR^d$, $d=2,3$. We mention, in particular, the works \cite{FM24, W13, CL03, I94, DF11, D10, BV80.1, BV80.2}. 

We are interested in studying the long-time existence of strong solutions to  \eqref{eq:Boussinesq} and \eqref{eq:nhEuler}, when the initial temperature $\theta_0=\theta(0)$ or the initial density $\rho_0=\rho(0)$ is close to a constant value, which we can take to be equal to $1$ without loss of generality. To this end, we assume that:
\begin{align}
   &\theta_0(x) = 1 +\ep\, \phi_0(x),  \label{eq:initialTemperature} \\
   &\rho_0(x)=  1 +\ep\, \varphi_0(x), \label{eq:initialDensity}
\end{align}
where $\ep>0$ is a small parameter. Specifically, we want to investigate how the time of existence of the solution depends on the parameter $\ep$.
We assume then $\bu_0$ or $\bv_0\in H^3_\sigma(\RR^2)$, $\theta_0$ or $\rho_0 \in H^3(\RR^2)$. ($H^3$ can be replaced by $H^s$ for $s>2$, but for simplicity we take an integer-valued exponent.)  By a strong solution, we mean a weak solution such that $\bu$ or $\bv\in C([0,T);H^3_\sigma\cap C^1([0,T);H_\sigma^{2})$, $\theta$ or $\rho\in C([0,T);H^3\cap C^1([0,T);H^{2})$, and the equation is satisfied at least pointwise a.e.\,.

Our goal is to show that we can extend the time of existence to any arbitrary $0<T<\infty$, provided $\ep$ is small enough. This result will be established by an improved energy estimate, Gr\"onwall's inequality and a continuation argument.
Indeed, as for the Euler equations,  the $H^3$ norm of $(\bu,\theta)$, for the Boussinesq equations, respectively  of $(\bv,\rho)$ for the density-dependent Euler equations, controls the time of existence of the solution; this can be easily seen from the proof of existence of strong solutions or from blow-up criteria, for both systems. We refer, in particular, to \cite[Theorem 3.1]{CN97} and to \cite[Theorem 1.1]{BLS20} (see also \cite{FZ21,CL03,Z10,MP19}), where $s=3$ for simplicity.

Our main results are the following theorems.

\begin{theo} \label{t:MainBoussinesq}
 Let $\bu_0\in H^3_\sigma(\RR^2)$ and $\phi_0\in H^2(\RR^2)$. Let $T\in (0,\infty)$. There exists $0<\ep_0\leq 1$  such that, if $\ep < \ep_0$, then there exists a unique $(\bu,\theta) \in C^1([0,T];H^3(\real^2))$ which is a strong solution of
 the 2D ideal Boussinesq equations \eqref{eq:Boussinesq} with $\bu(\cdot,0)=\bu_0$ and $\theta(\cdot,0)=\theta_0\equiv 1 + \ep \phi_0$.
\end{theo}

\begin{remark} \label{r:MainBoussinesq}
    We note that $\ep_0$ depends only on $T$ and $\widetilde{M}:= \norm{\bu_0}_{H^3} + \norm{\phi_0}_{H^3}$. Moreover, we will see from the proof that, given $\ep>0$, the time of existence $T$ has a lower bound of $C\log \log \log (1/\ep)$.
\end{remark}

\begin{theo} \label{t:MainNhEuler}
 Let $\bv_0\in H^3_\sigma(\RR^2)$ and $\varphi_0\in H^3(\RR^2)$, with $\varphi_0 \geq 0$. Let $T\in [0,\infty)$. There exists $0<\ep_0\leq 1$ such that, if $\ep < \ep_0$, then there exists a unique $(\bv,p,\rho) \in C^1([0,T];H^3(\real^2))$  which is a strong solution of the 2D density-dependent Euler equations \eqref{eq:nhEuler}   with initial data $\bv_0$, $\rho_0 = 1 + \ep \varphi_0$.
\end{theo}

\noindent We fix $T\in (0,\infty)$ throughout the rest of the paper.

\begin{remark} \label{r:DanchinProof}
A version of the first result was proved in \cite{D13}, and the second result is mentioned in \cite{DF11}, with a sketch of a proof. In both cases, the analysis is performed in  Besov spaces, and the estimates obtained for the time of existence are sharper. As a matter of fact, taking advantage of the secondary index in the definition of Besov spaces, it is possible to obtain a direct control of the $L^\infty$ norm of the vorticity, bypassing the need for logarithmic estimates, therefore gaining one less (negative) exponential dependence of $\epsilon$ on the desired time of existence of the solution, than in our results. However, the proof relies on techniques from Littlewood-Paley theory, while we employ more classical potential estimates that may be familiar to a wider audience.
\end{remark}

\subsection*{Acknowledgments} The authors thank Raph\"ael Danchin and Francesco Fanelli for useful discussion and for pointing out relevant literature. H. Bae was supported by the National Research Foundation of Korea (NRF) grant funded by the Korea government (MSIT) (grant No. RS-2024-00341870). M. C. Lopes Filho was partially supported by CNPq, through grant \# 304990/2022-1, and by FAPERJ, through  grant \# E-26/201.209/2021. A. Mazzucato was partially suppported by the US National Science Foundation grants DMS-1909103, DMS-2206453, DMS-2511023 and by Simons Foundation under Grant 1036502. H. J. Nussenzveig Lopes acknowledges the support of CNPq, through  grant \# 305309/2022-6, and that of FAPERJ, through  grant \# E-26/201.027/2022.

\section{The 2D ideal Boussinesq equations} \label{s:Boussinesq}

This section is devoted to the proof of Theorem \ref{t:MainBoussinesq}. We begin by recalling \cite[Theorem 2.1]{CN97}, where it was proved that there exists a time $T^\ast>0$ such that, given $\bu_0$, $\theta_0  \in H^3(\real^2)$, there exists a unique strong solution to \eqref{eq:Boussinesq}, $(\bu, \theta) \in C^0([0,T^\ast);H^3(\real^2))$,  with initial data $(\bu_0,\theta_0)$. Furthermore, it can be easily verified from the proof that, in fact, $(\bu, \theta) \in C^1([0,T^\ast);H^2(\real^2))$. It follows from the construction of the local-in-time solution $(\bu, \theta)$  that, as long as $\|\bu\|_{H^3} + \|\theta\|_{H^3}$ remains bounded, the solution may be prolonged further in time.

Fix $\bu_0$, $\phi_0 \in H^3(\real^2)$. Let $\ep > 0$ and take $\theta_0=1+\ep \phi_0$. Despite the fact that $\theta_0 \notin H^3(\real^2)$ it is still possible to find a strong solution to \eqref{eq:Boussinesq} with $\theta(\cdot,0)=\theta_0$. To this end we let $(\bu^\ep,\psi^\ep) \in C^0([0,T^\ast);H^3(\real^2)) \cap C^1([0,T^\ast);H^2(\real^2))$ be the solution to \eqref{eq:Boussinesq} obtained in \cite[Theorem 2.1]{CN97}, with initial data $(\bu_0,\ep\phi_0)$, on a time interval $[0,T^\ast]$. Note that, if
\beq \label{eq:Temperature}
  \theta^\ep (t,x) = 1+ \psi^\ep (t,x),
\eeq
then $(\bu^\ep,\theta^\ep)$ is a solution of \eqref{eq:Boussinesq} with initial data $(\bu_0,1+\ep \phi_0)$. Indeed, if $q^\ep=q^\ep(t,x)$ is the pressure for the solution $(\bu^\ep,\psi^\ep)$ then $(\bu^\ep,\theta^\ep)$ satisfies \eqref{eq:Boussinesq} with pressure $p^\ep(t,x) = q^\ep(t,x) +  x_2$.

Let us write $\psi^\ep \equiv \ep \phi^\ep$ and note that $\phi^\ep$ satisfies the same transport equation as $\psi^\ep$.

We now introduce
\begin{equation} \label{eq:contnormBouss}
    y^\ep(t):=  \norm{\bu^\ep(t,\cdot)}_{H^3} +\ep\, \norm{\phi^\ep(t,\cdot)}_{H^3} > 0.
\end{equation}
We refer to $y^\ep$ as the {\em continuation norm} for \eqref{eq:Boussinesq}. Our objective is to show that, for sufficiently small $\ep$, the continuation norm remains bounded on $[0,T]$ and, thus, ($\bu^\ep$, $\phi^\ep$) can be prolonged for additional time. For convenience we drop the superscript $\ep$ in $\bu^\ep$, $\phi^\ep$ and $y^\ep$, keeping $\ep$ only as a factor of $\|\phi^\ep\|_{H^3}$.

Recall that, for $x>0$,  $\log^+ x =\max\{\log x, 0\}$. Additionally, let $\omega=\curl \bu \equiv \nabla^\perp \cdot \bu$ denote the scalar vorticity of the flow $\bu$.

Our point of departure are the following estimates for strong solutions to \eqref{eq:Boussinesq}:
\begin{align}
    &\frac{dy}{dt}
    \leq C \left(1+\norm{D \bu}_{L^\infty} +\ep\, \norm{\nabla\phi}_{L^\infty}\right)
    y, \label{eq:BoussinesqEnergyEst.1} \\ \nonumber \\
    & \norm{D \bu}_{L^\infty} \leq C\, \left[  1+ \left(1+\log^+ \norm{\bu}_{H^3} \right) \norm{\omega}_{L^\infty}+\norm{\omega}_{L^2}\right], \label{eq:BoussinesqEnergyEst.2} \\  \nonumber \\
    & \norm{\nabla \phi}_{L^p} \leq \norm{\nabla\phi_0}_{L^p} \, \exp\left( \int_0^t \norm{D \bu}_{L^\infty}\right), \label{eq:BoussinesqEnergyEst.4} \\ \nonumber \\
     & \norm{\omega}_{L^p} \leq \norm{\omega_0}_{L^p} + C \ep\, \int_0^t
    \norm{\nabla\phi(\tau)}_{L^p}\, d\tau, \label{eq:BoussinesqEnergyEst.3}
\end{align}
valid for any $p\in [1,\infty]$ and for $t\in [0,T]$. Estimate \eqref{eq:BoussinesqEnergyEst.1} is established in a standard way, by performing the same energy estimates used to derive \cite[(2.4)]{CN97}, which relied on calculus inequalities in Sobolev spaces, see for instance \cite[Lemma 3.4]{MajdaBertozzi}. Estimate \eqref{eq:BoussinesqEnergyEst.2} is a potential theory estimate, derived in two space dimensions in \cite{Kato1986} and then in three space dimensions in \cite{BKM1984}, see also \cite[Proposition 3.8]{MajdaBertozzi}. Estimate \eqref{eq:BoussinesqEnergyEst.4} is a consequence of the equation satisfied by $\nabla \phi$, namely,
\[\partial_t \nabla \phi + (u \cdot \nabla) \nabla \phi = - D \bu \nabla \phi,\]
where $D\mathbf{u}$ is the Jacobian matrix for $\mathbf{u}$.
Finally, \eqref{eq:BoussinesqEnergyEst.3} follows using energy methods on the vorticity equation:
\[\partial_t \omega + u \cdot \nabla \omega = \ep \nabla^{\perp}\phi \cdot \be_2.\]

We use \eqref{eq:BoussinesqEnergyEst.4} and \eqref{eq:BoussinesqEnergyEst.3} in \eqref{eq:BoussinesqEnergyEst.2} to find:
\begin{align} \label{eq:NablaBuEst}
     \norm{D \bu}_{L^\infty}  & \leq  C\, \Bigg\{  1+(1+\log^+ y) \left[1+C\ep \, \int^t_0 \exp\left(\int_0^s y \, d\tau\right) \, ds \right] \\
    \nonumber & + C\ep \, \int^t_0 \exp\left(\int_0^s y \, d\tau\right) \, ds \Bigg\},
\end{align}
where,  in \eqref{eq:BoussinesqEnergyEst.4}, we first estimated $\|D \bu\|_{L^\infty}$ by $\|u\|_{H^3}$, and subsequently by $y$.

Substituting the inequality \eqref{eq:NablaBuEst} into \eqref{eq:BoussinesqEnergyEst.1} gives an integro-differential inequality for $y$:
\beq \label{eq:yIneqBoussinesq}
   \frac{d y}{d t} \leq  C\, y \,\left[ (1+\log^+ y) + \ep \,(1+\log^+ y)\,\int^t_0 \exp \left(\int_0^s y \, d\tau\right)\, ds + \ep \, \exp \left(\int_0^t y \, d\tau\right)\right].
\eeq

To proceed with our analysis of \eqref{eq:yIneqBoussinesq} we first consider the integro-differential equation below:

\beq \label{eq:yODEBoussinesq}
   \frac{d z}{d t} = C\, z \,\left[  (1+\log z) + \ep \,(1+\log z)\,\int^t_0 \exp \left(\int_0^s z \, d\tau \right)\, ds + \ep \, \exp \left(\int_0^t z \, d\tau \right) \right],
\eeq
supplemented by an initial condition $z(0)=z_0>1$.

We note that the nonlinearity on the right-hand-side of \eqref{eq:yODEBoussinesq} is the same as that on the right-hand-side of \eqref{eq:yIneqBoussinesq}, substituting $\log^+$ by $\log$.

Dividing both sides of \eqref{eq:yODEBoussinesq} by $z$, making the change of variables \ $x:= 1+ \log z$, and consolidating all constant terms, yields:
\beq \label{eq:yODEBoussinesqFinal}
    \dt{x} =  C\,\left[ x + \ep \, x \int^t_0 \exp\left( b\int_0^s e^x \, d\tau \right)\, ds
    + \ep\, \exp\left( b\int_0^t e^x \, d\tau \right) \right],
\eeq
with $b=e^{-1}$ and initial data $x_0=1+ \log z_0 $.
We aim at showing that, for sufficiently small $\ep$, the solution $x(t)$ exists for $t\in [0,T)$.

 We note that the right-hand side of  \eqref{eq:yODEBoussinesqFinal} is a non-linear perturbation of a linear term. To emphasize this point, we
use an integrating factor and we introduce $\ds w(t):=e^{-C\,t}\,x(t)$; we write this equation in compact form as:
\beq \label{eq:yODEBoussinesqReduced}
  \dt{w} = \ep \, F(t,w),
\eeq
with $F$ given by
\begin{align}
    &F= F(t,w) :=  C\left\{  w(t) \int^t_0 \exp\left[b\int_0^s \exp\left(w(\tau)\,e^{C\, \tau} \right) \, d\tau\right]\, ds \right. \nonumber\\
    & \qquad \qquad \left.+ e^{-C\,t} \exp\left[b\int_0^t \exp\left(w(\tau)\,e^{C\, \tau} \right) \, d\tau\right] \right\}. \nonumber
\end{align}
Note that the initial data $w(0)=w_0$ satisfies $w_0>1$. In addition, we emphasize that \eqref{eq:yODEBoussinesqReduced} is not an ordinary differential equation since the nonlinearity $F$ depends on the entire history of $w$ up until time $t$.

Equivalently, in mild form, we may write \eqref{eq:yODEBoussinesqReduced} as:
\beq \label{eq:yODEBoussinesqMild}
   w(t)= w_0 + \ep \, \int^t_0 F(\tau,w)\, d\tau.
\eeq

Let $\cX:=C([0,T))$. We introduce a nonlinear map $\Psi_{w_0}: \cX \to \cX$ defined by $w \mapsto \Psi_{w_0}[w]$, where $\Psi_{w_0}[w]$ is precisely the right-hand-side of \eqref{eq:yODEBoussinesqMild}.

 We will solve \eqref{eq:yODEBoussinesqMild} via Banach Contraction Mapping Theorem in the space $\cX$ with the uniform norm (see e.g. \cite[Theorem 3.1]{MajdaBertozzi}).

To apply the Banach Contraction Mapping Theorem, we need to establish that, for $\ep$ sufficiently small, the following hold true:
\begin{enumerate}[label={\bf CP \arabic*.}, ref=\textcolor{black}{\bf CP \arabic*}]
 \item \label{i:CP1} There exists a ball $B(0,M)\subset \cX$ such that $\Psi_{w_0}: B(0,M)\to B(0,M)$ is a continuous map.

 \item \label{i:CP2} $\Psi_{w_0}$ is a contraction on $B(0,M)$.
\end{enumerate}

\begin{proof}[Proof of \ref{i:CP1}]
Set $\widetilde{M}= w_0 $ and choose $M=2\widetilde{M}$. It is immediate from \eqref{eq:yODEBoussinesqMild}  that $\Psi_{w_0}[w]\in \cX$. Furthermore, if $w\in B(0,M)$, then $\Psi_{w_0}[w] \in B(0, M)$   as long as the following   inequality is satisfied:
\[
  \widetilde{M} + C\,T\,\ep (1+ 2 T \widetilde{M})\, \exp\left(T\,\exp(2\widetilde{M}\,e^{C\,T})\right) \leq 2\widetilde{M},
\]
or equivalently,
\beq  \label{eq:McondBoussinesq}
   \ep \leq \frac{M}{2C \,T\, (1+ T M)\, \exp\left(T\,\exp(M\,e^{C\, T})\right)}\equiv \ep_1.
\eeq
 Continuity of $w \mapsto \Psi_{w_0}[w]$ is immediate.
\end{proof}

\begin{proof}[Proof of \ref{i:CP2}]
We begin by showing that $F$ is locally Lipschitz in $w$.
For notational convenience, we set:
\[
    g(t,w):= C \,\exp\left[b\int_0^t \exp\left( w(\tau)\,e^{C\tau}\right)\,d\tau\right].
\]
Let $w,v\in B(0,M)$. Then
\begin{align} \label{FBoussLipest}
    &F(t,w)-F(t,v) = w(t) \int^t_0 g(s,w)\, ds -v(t) \int^t_0 g(s,v)\, ds + \frac{1}{C}e^{-Ct}[g(t,w)- g(t,v)]
     \nonumber \\
     & = \big(w(t)-v(t)\big) \int^t_0 g(s,w)\, ds + v(t)  \int^t_0 \left[g(s,w) -g(s,v)\right]\, ds + \frac{1}{C}e^{-Ct}[g(t,w)- g(t,v)].
\end{align}

We will estimate the difference $\norm{F(t,w)-F(t,v)}_{L^\infty}$ in terms of $\norm{w-v}_{\infty}$.
First, we note that
\begin{equation} \label{gest1}
0< g(t,w) \leq C\exp[bT\exp(Me^{CT})].
\end{equation}

Next, we must estimate the difference $g(t,w)-g(t,v)$. We introduce
\[\xi(t)=\int_0^t \exp[w(\tau)e^{C\tau}]\,d\tau \text{ and } \eta(t)=\int_0^t \exp[v(\tau)e^{C\tau}]\,d\tau.\]
With this notation we find:
\begin{align*}
|g(t,w)-g(t,v)| & \leq C b|\xi(t)-\eta(t)|\, \sup_{0\leq s \leq 1} \, \exp[b(s \xi(t) + (1-s)\eta(t))] \\
&\leq Cb \, \exp[bT\exp(Me^{CT})]\,|\xi(t)-\eta(t)|.
\end{align*}
We now estimate $|\xi(t)-\eta(t)|$:
\begin{align*}
    |\xi(t)-\eta(t)| & =\left| \int_0^t \left[ \exp(w(\tau)e^{C\tau}) - \exp(v(\tau)e^{C\tau}) \right] \,d\tau \right| \\
& = \left| \int_0^t \int_0^1 \frac{d}{ds} \exp[(s w(\tau) + (1-s) v(\tau))e^{C\tau}]\, ds \,d\tau \right|\\
& = \left| \int_0^t \int_0^1 \exp[(s w(\tau) + (1-s) v(\tau))e^{C\tau}] \, e^{C\tau} \, (w(\tau) - v(\tau)) \, ds \,d\tau \right|\\
& \leq Te^{CT} \exp(Me^{CT})\,\norm{w-v}_{L^\infty}.
\end{align*}
Hence,
\begin{equation} \label{gest2}
  |g(t,w)-g(t,v)| \leq Cb \, \exp[bT\exp(Me^{CT})] \,Te^{CT} \, \exp(Me^{CT})\,\norm{w-v}_{L^\infty}.
\end{equation}
We insert \eqref{gest1}, \eqref{gest2} into \eqref{FBoussLipest} to obtain
\begin{equation} \label{eq:LipcondBoussinesq}
  \norm{F(t,w)-F(t,v)}_{\infty} \leq  L \norm{w-v}_{\infty},
\end{equation}
with
\[L= \exp[bT\exp(Me^{CT})]\, [CT+CbMT^2e^{CT}\exp(Me^{CT})\,+\,bTe^{CT}\exp(Me^{CT})].\]

We use \eqref{eq:LipcondBoussinesq} to show that, for $\varepsilon$ sufficiently small, $\Psi_{w_0}$ is indeed a contraction on $B(0,M)$:
\[\left|\Psi_{w_0}[w]-\Psi_{w_0}[v]\right| =\varepsilon \left|\int_0^T (F(\tau,w)-F(\tau,v))\,d\tau \right| \leq \varepsilon TL\norm{w-v}_{L^\infty}. \]
Therefore $\Psi_{w_0}$ is a contraction on $B(0,M)$ as long as
\[\varepsilon < \frac{1}{TL},\]
that is,
\beq \label{eq:ContractionCondBoussinesq}
 \ep < \{T\exp[bT\exp(Me^{CT})]\, [CT+bTe^{CT}\exp(Me^{CT})(CMT +1)]\}^{-1} \equiv \ep_2.
\eeq

\end{proof}

To conclude the existence and uniqueness of the solution $w(t)$ of \eqref{eq:yODEBoussinesqReduced}, we choose
\beq \label{eq:ep0}
\ep_0 = \min\{1,\ep_1,\ep_2\},
\eeq
where $\ep_1$ was defined in \eqref{eq:McondBoussinesq} and $\ep_2$ was given in \eqref{eq:ContractionCondBoussinesq}.
Hence, writing $x(t) = e^{Ct}w(t)$ and $z(t) = e^{x(t)-1}$ we obtain, under the smallness condition $\varepsilon < \ep_0$, the existence and uniqueness of a solution to \eqref{eq:yODEBoussinesq}
on the time interval $[0,T)$. In summary, we have established the following result.

\begin{prop} \label{prop:BoussinesqODEexist}
Fix $T>0$ and $z_0 >1$. Then there exists $0<\varepsilon_0 \leq 1$ such that, for all $0<\varepsilon < \varepsilon_0$, there exists one and only one solution $z^\ep=z^\ep(t) \in C^1([0,T))$ to \eqref{eq:yODEBoussinesq} such that $z^\ep(0)=z_0$. Moreover, $\ep_0 = \ep_0(T,z_0)$.
\end{prop}

We are now ready to address the main topic of this section.

\begin{proof}[Proof of   \rm{Theorem \ref{t:MainBoussinesq}}:]

Let $T>0$. Fix  $\bu_0 \in H^3_\sigma (\real^2)$ and $\phi_0 \in H^3(\real^2)$. Let
\beq \label{eq:D}
D:=\|\bu_0\|_{H^3}+\|\phi_0\|_{H^3}.
\eeq

We will make use of Proposition \ref{prop:BoussinesqODEexist} with $z_0 = D+1$, so that $w_0=1+\log(D+1)$. Choose $\ep_0$ as in \eqref{eq:ep0} with $M=2(1+\log(D+1))$. Let $0<\ep < \ep_0$.

Let $T_{\mathrm{max}}^\ep$ denote the maximal time of existence for the solution of the Boussinesq equations \eqref{eq:Boussinesq} with initial data $\bu_0$, $\theta_0 = 1+\ep\phi_0$. We wish to show that $T \leq T_{\mathrm{max}}^\ep$ for all $0 < \ep < \ep_0$. Let us assume, by contradiction, that $T > T_{\mathrm{max}}^\ep$ for some $0 < \ep < \ep_0$. Fix such an $\ep$.

In particular, $T_{\mathrm{max}}^\ep < \infty$, so that, since $y^\ep=y^\ep(t)=\|\bu^\ep(t)\|_{H^3} + \varepsilon\|\phi^\ep(t)\|_{H^3}$, as introduced in \eqref{eq:contnormBouss}, is a continuation norm, it follows that
\beq \label{eq:yblowsup}
\limsup_{t\to (T_{\mathrm{max}}^\ep)^-} y^\ep(t) = +\infty.
\eeq

Let $z^\ep=z^\ep(t) \in C^1([0,T))$ be the solution to \eqref{eq:yODEBoussinesq} with $z^\ep(0)=D+1$, where $D$ was introduced in \eqref{eq:D}. Note that the right-hand-side of  \eqref{eq:yODEBoussinesq} is positive as long as $z^\ep > 1$. Since, initially, $z^\ep (0) > 1$, we have $z^\ep(t) > 1$ for all $t \in [0,T)$ and, hence, $z^\ep$ is increasing.

We claim that $y^\ep(t) < z^\ep(t) $ for all $0 \leq t < T_{\mathrm{max}}^\ep$. To see this, set $Y^\ep :=y^\ep - z^\ep$ and observe that $Y^\ep(0)\leq D-(D+1)=-1<0$. Assume, by contradiction, that there exists $0<t_2^\ep< T_{\mathrm{max}}^\ep$ such that $Y^\ep(t_2^\ep)\geq 0$.  Let
\beq \label{eq:t1ep}
t_1^\ep := \inf\,\{t\in (0,t_2^\ep] \,|\, Y^\ep (t) \geq 0\}.
\eeq

Clearly $t_1^\ep>0$ and, by continuity, $Y^\ep(t_1^\ep)=0$, so that $y^\ep(t_1^\ep)=z^\ep(t_1^\ep)>1$. In addition, $Y^\ep(t) <0$ for all $0\leq t < t_1^\ep$. Next, set
\beq \label{eq:t0ep}
t_0^\ep := \inf \,\{a\in [0,t_1^\ep) \,|\, y^\ep (t) >1 \text{ for all } t \in (a,t_1^\ep]\}.
\eeq

For notational convenience  we will denote the right-hand side of \eqref{eq:yODEBoussinesq} by $\widetilde{F}_{\ep}(z,t)$. Observe  that $\widetilde{F}_{\ep}$ is monotone non-decreasing with respect to $z$, that is, if $y(\tau)\leq z(\tau)$ {\em for all $0\leq \tau  \leq t$} then $\widetilde{F}_{\ep}(y,t) \leq \widetilde{F}_{\ep}(z,t)$.

Note that, with this notation, on the interval $(t_0^\ep, t_1^\ep)$ we have
\beq \label{eq:yzInEqualBoussinesq}
    \dt{(y^\ep - z^\ep)} \leq  \widetilde{F}_\ep(y^\ep,t) - \widetilde{F}_\ep(z^\ep,t) ,
\eeq
since $\log^+y^\ep (t)=\log y^\ep (t)$ for $t$ in this time interval. In addition, observe that, because $y^\ep(\tau)< z^\ep(\tau)$ for all $0\leq \tau<t_1^\ep$, it follows that
\beq \label{eq:dYepdtNONPOS}
\dt{Y^\ep} (t) = \dt{(y^\ep - z^\ep)} (t) \leq 0
\eeq
for all $t \in (t_0^\ep,t_1^\ep)$. However, we must also have, by the Mean Value Theorem, that, for some $t_0^\ep < t_\ast^\ep < t_1^\ep$
\beq \label{eq:dYepdtPOS}
\dt{Y^\ep} (t_\ast^\ep) = \frac{Y^\ep(t_1^\ep) - Y^\ep(t_0^\ep)}{t_1^\ep - t_0^\ep}.
\eeq
But \eqref{eq:dYepdtPOS} is in contradiction with \eqref{eq:dYepdtNONPOS} since $Y^\ep(t_1^\ep)=0$ and $Y^\ep(t_0^\ep)<0$. This establishes our claim.

Next, having shown that $y^\ep(t) < z^\ep(t) $ for all $0 \leq t < T_{\mathrm{max}}^\ep$ we may deduce that
\[\limsup_{t\to (T_{\mathrm{max}}^\ep)^-} y^\ep(t) \leq \limsup_{t\to (T_{\mathrm{max}}^\ep)^-} z^\ep (t) < +\infty,\]
since the time of existence of $z^\ep$ is at least $T>T_{\mathrm{max}}^\ep$. This contradicts \eqref{eq:yblowsup} and hence we must have $T\leq T_{\mathrm{max}}^\ep$ for all $0<\ep < \ep_0$. This concludes the proof.

\end{proof}

\begin{remark}\label{r:EpsilonGrowth}
From \eqref{eq:McondBoussinesq} and \eqref{eq:ContractionCondBoussinesq}, it follows that \[
  \ep_0\sim e^{-e^{e^{T}}}. 
\]
The same proof of existence also shows that, if we fix $\ep$ instead of $T$, then we have time of existence $T$ with satisfies:
\[
    T\sim \log\log\log\frac{1}{\ep}.
\]
\end{remark}


\section{The 2D density-dependent Euler equations} \label{s:nhEuler}

\bigskip

We now turn to the proof of Theorem \ref{t:MainNhEuler}. Since the argument is similar to that of the proof of Theorem \ref{t:MainBoussinesq}, we give details only when the proof deviates significantly.
Again, the proof is based on energy estimates and a continuation criterion for strong solutions. In our analysis we will rely heavily on {\it a priori} estimates derived in \cite{BLS20}, but with slight modifications.

Fix $\bv_0,$ $\varphi_0 \in H^3(\real^2)$, with $\varphi_0 \geq 0$. We are interested in solving \eqref{eq:nhEuler} with initial data $\bv(0,\cdot)=\bv_0$ and $\rho(0,\cdot)=1+\ep \varphi_0$. We will need an additional hypothesis, namely, $\eta_0 := \displaystyle{\frac{1}{\rho(0,\cdot)} - 1 \in H^3(\real^2)}$, which is certainly satisfied if
\begin{equation} \label{eq:firstep0cond}
    0\leq \ep \leq \min \left\{ 1, \frac{1}{2\max \varphi_0} \right\}.
\end{equation} Assuming these conditions we  use \cite[Theorem 1.1]{BLS20} to find a unique strong solution $\bv^\ep$, $\rho^\ep - 1 \in C^0([0,T^\ast);H^3(\real^2))$. It is standard to show that, additionally, $\bv^\ep$, $\rho^\ep - 1 \in  C^1([0,T^\ast);H^{2}(\real^2))$. Moreover,
\[
\eta^\ep = \displaystyle{\frac{1}{\rho^\ep} - 1} \in C^0([0,T^\ast);H^3(\real^2)) \cap C^1([0,T^\ast);H^{2}(\real^2)).
\]
It was established in \cite[Theorem 1.1]{BLS20} that as long as
\beq \label{eq:Mcontnorm}
\mathcal{M^\ep} (t)\equiv \|\varphi^\ep(t,\cdot)\|_{H^3}+\|\eta^\ep(t,\cdot)\|_{H^3}+\|\bv^\ep(t,\cdot)\|_{H^3}
\eeq
remains bounded, the solution of \eqref{eq:nhEuler} can be further extended in time.

As we did for the Boussinesq equations, we write:
\beq \label{eq:Density}
   \rho^\ep(t,x)= 1 + \ep \,\varphi^\ep (t,x),
\eeq
with $  \varphi^\ep(0,x)=\varphi_0(x)$. We note that $\varphi^\ep$ is also transported by $\bv^\ep$.

Let us also recall the basic energy estimate for \eqref{eq:nhEuler}, that is, the conservation of kinetic energy:
\beq \label{eq:basicenestnhEuler}
\|\sqrt{\rho^\ep(t,\cdot)}\,\bv^\ep(t,\cdot)\|_{L^2}=\|\sqrt{1+\ep\varphi_0}\,\bv_0\|_{L^2}.
\eeq

Our analysis is based on a useful differential inequality, which we will proceed to derive. Similarly to what we did in Section \ref{s:Boussinesq}, we will omit the superscript $\ep$ in $\bv^\ep$, $\rho^\ep$, $\eta^\ep$, and $\varphi^\ep$, for convenience. Differently from the case of Boussinesq, where the pressure does not directly play a role in the continuation argument, for the density-dependent Euler equations it turns out that $\nabla p/\rho$ enters into the vorticity equation.

We will derive a differential inequality for $\|\bv\|_{H^3}$. Let us begin by recalling the $2D$ vorticity equation:
\beq \label{eq:vorteqetaform}
\partial_{t}\upomega +\bv\cdot \nabla \upomega=-\nabla^{\perp}\eta\cdot \nabla p,
\eeq
where $\upomega = \nabla^\perp \cdot \bv$ is the vorticity and $\eta = \displaystyle{\frac{1}{\rho}-1}$.

From this, we easily find, by energy methods,
\[
\begin{split}
\frac{d}{dt} \left\|\upomega \right\|_{L^{2}}&\leq C\left\|\nabla^{\perp}\eta\cdot \nabla p\right\|_{L^{2}}, \qquad \text{and}\\
 \frac{d}{dt} \left\|\nabla \upomega \right\|_{L^{2}}&\leq C\left(\left\|\nabla \bv\right\|_{L^{\infty}}   \left\|\nabla \upomega\right\|_{L^{2}}+ \left\|\nabla (\nabla^{\perp}\eta\cdot \nabla p)\right\|_{L^{2}}\right).
\end{split}
\]

Recall that $\nabla \bv$ is related to $\upomega$ by a zero-th order singular integral operator. Therefore, from the  Calder\'on-Zygmund inequality \cite{CalderonZygmund1952}, we have for any $1<p<\infty$,
\beq \label{eq:CZineqvomega}
\| D^2 \bv \|_{L^p} \leq C \|\nabla \upomega\|_{L^p},
\eeq
where $D^2$ refers to any second (spatial) derivative of $\bv$.

Using energy estimates and \eqref{eq:CZineqvomega} with $p=4$ we obtain
\[
\frac{d}{dt} \left\|\Delta \upomega \right\|_{L^{2}}\leq C\left\|\nabla \bv\right\|_{L^{\infty}}   \left\|\Delta\upomega\right\|_{L^{2}}+ \left\|\Delta(\nabla^{\perp}\eta\cdot \nabla p)\right\|_{L^{2}}+C \left\|\nabla \upomega\right\|^{2}_{L^{4}}.
\]
We use the Gagliardo-Nirenberg interpolation inequality \cite{Gagliardo1959,Nirenberg1959} (see also \cite{BrezisMironescu2018} and references therein) to deduce that
\[
\left\|\nabla \upomega\right\|^{2}_{L^{4}} \leq C \left\|\upomega\right\|^{2}_{W^{1,4}} \leq C \left\|\upomega\right\|_{L^{\infty}}\left\|\upomega\right\|_{H^{2}} \leq C\left\|\nabla \bv\right\|_{L^{\infty}}\left\|\upomega\right\|_{H^{2}}.
\]
Putting together this collection of estimates we obtain:
\[
\begin{split}
\frac{d}{dt} \left\|\upomega \right\|_{H^{2}}&\leq C\left\|\nabla \bv\right\|_{L^{\infty}} \left\|\upomega\right\|_{H^{2}}+ \left\|\nabla^{\perp}\eta\cdot \nabla p\right\|_{H^{2}}\\
& \leq C\left\|\nabla \bv\right\|_{L^{\infty}} \left\|\upomega\right\|_{H^{2}}+ C\left\|\nabla\eta\right\|_{L^{\infty}}\left\|\nabla p\right\|_{H^{2}}+C\left\|\nabla p\right\|_{L^{\infty}}\left\|\nabla \eta\right\|_{H^{2}},
\end{split}
\]
i.e.,
\beq \label{eq:omegaH2est}
\frac{d}{dt} \left\|\upomega \right\|_{H^{2}} \leq C\left\|\nabla \bv\right\|_{L^{\infty}} \left\|\upomega\right\|_{H^{2}}+   C\left\|\nabla p\right\|_{H^{2}}\left\|\eta\right\|_{H^{3}}.
\eeq

Next, we write the equation satisfied by $\bv$ \eqref{eq:nhEuler}$_1$ as
\[\partial_t \bv + \bv \cdot \nabla \bv = -\frac{1}{\rho} \nabla p,\]
and we recall $\eta = \displaystyle{\frac{1}{\rho}-1}$.  Energy estimates yield
\[\frac{d}{dt} \| \bv \|_{L^2} \leq \|(\eta +1) \nabla p\|_{L^2} \leq \left(\|\eta\|_{L^\infty}+1 \right)
\|\nabla p\|_{L^2},\]
so that
\beq \label{eq:bvL2est}
\frac{d}{dt} \| \bv \|_{L^2} \leq C\left( \|\eta\|_{H^3} + 1\right) \|\nabla p\|_{H^2}.
\eeq

Putting together estimates \eqref{eq:omegaH2est} and \eqref{eq:bvL2est}, we find
\beq \label{eq:vH3Bound}
\frac{d}{dt} \left\|\bv \right\|_{H^{3}} \leq C  \left\|\nabla \bv\right\|_{L^{\infty}} \left\|\bv\right\|_{H^{3}}+ C\left\|\nabla p\right\|_{H^{2}}\left( \left\|\eta\right\|_{H^{3}} + 1 \right).
\eeq

Let us now estimate the pressure $p$. First recall that
\[
\left\|\nabla p\right\|_{H^{2}} \leq C \left( \left\|\nabla p\right\|_{L^{2}}+\left\|\nabla \Delta p\right\|_{L^{2}} \right).
\]

We note that $p$ satisfies the elliptic equation:
\[
    -\dive \left(\frac{\nabla p}{\rho} \right) = \dive(\bv \cdot \nabla \bv).
\]

Since $\varphi$ satisfies the linear transport equation \eqref{eq:nhEuler}$_2$ with the divergence-free Lipschitz velocity field $\bv$, it follows that $\max \varphi = \max \varphi_0$. Consequently, since we assumed that $\ep$ satisfies  \eqref{eq:firstep0cond}, we have, for  $1/\rho$:
\beq \label{eq:etaPointwiseBound}
   \frac{2}{3} \leq  \frac{1}{\rho}\equiv  \frac{1}{1+\ep\,\varphi(t,x)} \leq 1.
\eeq

Energy methods together with the estimate \eqref{eq:etaPointwiseBound} on $1/\rho$ give
\beq \label{eq:NablapL2Bound}
  \norm{\nabla p}_{L^2} \leq \frac{3}{2}\, \norm{\bv}_{L^2} \norm{\nabla \bv}_{L^\infty}.
\eeq
Now, using again \eqref{eq:etaPointwiseBound} together with the conservation of energy \eqref{eq:basicenestnhEuler} yields
$\norm{\bv}_{L^2}$ bounded, so that
\beq \label{eq:NablapL2Boundprime}
\norm{\nabla p}_{L^2} \leq C \norm{\nabla \bv}_{L^\infty}.
\eeq


Next, using \cite[Eq. (2.13)]{BLS20} we have
\[-\dive \left(\frac{1}{\rho} \Delta \nabla p \right ) = \Delta \, \dive (\bv \cdot \nabla \bv) + \dive \, [\Delta \, (\eta \nabla p) - \eta \, \Delta \nabla p].
\]
Therefore, multiplying by $\Delta p$, integrating by parts and using \eqref{eq:etaPointwiseBound} we arrive at
\[
\| \nabla \Delta p \|_{L^{2}}  \leq  \|\nabla \, \dive (\bv \cdot \nabla \bv) \|_{L^2} +
\|\Delta \, (\eta \nabla p) - \eta \, \Delta \nabla p \|_{L^2} .
\]
It is easy to see that $\nabla \, \dive (\bv \cdot \nabla \bv)$ is a sum of terms of the form $D^2 \bv D \bv$. Thus, using additionally \eqref{eq:CZineqvomega} with $p=2$, we obtain
\[\left\| \nabla \Delta p \right\|_{L^{2}} \leq C \left\|\nabla \bv\right\|_{L^{\infty}} \left\| \nabla \upomega\right\|_{L^{2}}+ C\left\|\nabla\eta\right\|_{L^{\infty}}\left\|\Delta p\right\|_{L^{2}}+C\left\|\nabla p\right\|_{L^{\infty}}\left\|\Delta \eta\right\|_{L^{2}}.
\]

Thus, we obtain
\[
\begin{split}
\left\|\nabla p\right\|_{H^{2}} &\leq C\left\|\nabla \bv\right\|_{L^{\infty}} + C \left\|\nabla \bv\right\|_{L^{\infty}} \left\| \nabla \upomega\right\|_{L^{2}}+ C\left\|\nabla\eta\right\|_{L^{\infty}}\left\|\Delta p\right\|_{L^{2}}+C\left\|\nabla p\right\|_{L^{\infty}}\left\|\Delta \eta\right\|_{L^{2}}\\
&\leq C\left\|\nabla \bv\right\|_{L^{\infty}} + C \left\|\nabla \bv\right\|_{L^{\infty}} \left\| \nabla \upomega\right\|_{L^{2}} +C\left(\left\|\Delta p\right\|_{L^{2}} + \left\|\nabla p\right\|_{L^{\infty}}\right)\left\|\eta\right\|_{H^{3}}.
\end{split}
\]
From the Gagliardo-Nirenberg inequalities we see that
\[
\left\|\Delta p\right\|_{L^{2}} +\left\|\nabla p\right\|_{L^{\infty}}
\leq
C \left\|\nabla p\right\|^{\frac{1}{2}}_{L^{2}}\left\|\nabla \Delta p\right\|^{\frac{1}{2}}_{L^{2}}.
\]
Therefore, using \eqref{eq:NablapL2Boundprime} followed by Young's inequality we have
\[
\begin{split}
\left\|\Delta p\right\|_{L^{2}} +\left\|\nabla p\right\|_{L^{\infty}}
& \leq C \left\|\nabla \bv\right\|^{\frac{1}{2}}_{L^{\infty}}\left\|\nabla \Delta p\right\|^{\frac{1}{2}}_{L^{2}}\\
&\leq C\lambda \left\|\nabla p\right\|_{H^{2}} +\frac{1}{4\lambda}  \left\|\nabla \bv\right\|_{L^{\infty}},
\end{split}
\]
for any $\lambda > 0$.

Hence  we derived the estimate
\[
\left\|\nabla p\right\|_{H^{2}} \leq C\left\|\nabla \bv\right\|_{L^{\infty}} + C \left\|\nabla \bv\right\|_{L^{\infty}} \left\| \nabla \upomega\right\|_{L^{2}}
+C \left(\lambda \left\|\nabla p\right\|_{H^{2}} +\frac{1}{4\lambda}  \left\|\nabla \bv\right\|_{L^{\infty}}\right)\left\|\eta\right\|_{H^{3}}.
\]
By choosing $\lambda$ such as
\[
C\lambda\left\|\eta\right\|_{H^{3}}=\frac{1}{2},
\]
we obtain
\beq \label{eq:pH2Bound}
\left\|\nabla p\right\|_{H^{2}} \leq C\left( \left\|\nabla \bv\right\|_{L^{\infty}} + \left\|\nabla \bv\right\|_{L^{\infty}} \left\| \nabla \upomega\right\|_{L^{2}} + \left\|\nabla \bv\right\|_{L^{\infty}}\left\|\eta\right\|_{H^{3}}^2 \right).
\eeq
Using \eqref{eq:pH2Bound} in \eqref{eq:vH3Bound} we find

\begin{align} \label{eq:dvH3dt1}
\frac{d}{dt} \|\bv \|_{H^{3}} &  \leq C\left\|\nabla \bv\right\|_{L^{\infty}} \left\|\bv\right\|_{H^{3}} + \\
& + C \left( \left\|\nabla \bv\right\|_{L^{\infty}} + \left\|\nabla \bv\right\|_{L^{\infty}} \left\| \nabla \upomega\right\|_{L^{2}} + \left\|\nabla \bv\right\|_{L^{\infty}}\left\|\eta\right\|_{H^{3}}^2 \right)
\left( \left\|\eta\right\|_{H^{3}} + 1 \right) \nonumber \\
& \leq C\left\|\nabla \bv\right\|_{L^{\infty}} \left\|\bv\right\|_{H^{3}}
+ C \left\|\nabla \bv\right\|_{L^{\infty}} +
C\left\|\nabla \bv\right\|_{L^{\infty}}\left\|\eta\right\|_{H^{3}}^2 \nonumber \\
& +
C\left\|\nabla \bv\right\|_{L^{\infty}}\left\|\eta\right\|_{H^{3}} + C \left\|\nabla \bv\right\|_{L^{\infty}} \left\| \nabla \upomega\right\|_{L^{2}}\left\|\eta\right\|_{H^{3}} +  C \left\|\nabla \bv\right\|_{L^{\infty}}\|\eta \|_{H^{3}}^3 \nonumber \\
& \leq C \left\|\nabla \bv\right\|_{L^{\infty}} \left[ 1 + \| \bv \|_{H^3} + \| \eta \|_{H^3}
 + \| \eta \|_{H^3}^2  + \| \eta \|_{H^3}^3 \right] + C \left\|\nabla \bv\right\|_{L^{\infty}} \|\nabla \upomega\|_{L^2} \| \eta \|_{H^3}. \nonumber
\end{align}

Using the Gagliardo-Nirenberg inequalities to estimate $\|\nabla \upomega\|_{L^2}$ we estimate the last term in
\eqref{eq:dvH3dt1}:

\begin{align*}
\left\|\nabla \bv\right\|_{L^{\infty}} \left\| \nabla \upomega\right\|_{L^{2}}\left\|\eta\right\|_{H^{3}}&\leq C\left\|\nabla \bv\right\|^{\frac{1}{3}}_{L^{\infty}} \|\bv \|^{\frac{1}{3}}_{L^{2}}\left\|\eta\right\|_{H^{3}} \left\|\nabla \bv\right\|^{\frac{2}{3}}_{L^{\infty}}\left\| \nabla \Delta \bv\right\|^{\frac{2}{3}}_{L^{2}}\\
&\leq C \left\|\nabla \bv\right\|_{L^{\infty}} \left\|\eta\right\|^{3}_{H^{3}}+C \left\|\nabla \bv\right\|_{L^{\infty}}\left\| \nabla \Delta \bv\right\|_{L^{2}}\\
& \leq C \left\|\nabla \bv\right\|_{L^{\infty}} \left\|\eta\right\|^{3}_{H^{3}}+C \left\|\nabla \bv\right\|_{L^{\infty}}
\left\|\bv\right\|_{H^{3}},
\end{align*}
where we used Young's inequality and the conservation of energy \eqref{eq:basicenestnhEuler} in the penultimate inequality.

With this \eqref{eq:dvH3dt1} becomes
\beq \label{eq:dvH3dt2}
\frac{d}{dt} \left\|\bv \right\|_{H^{3}}  \leq C \left\|\nabla \bv\right\|_{L^{\infty}} \left[ 1 + \| \bv \|_{H^3} + \| \eta \|_{H^3} + \| \eta \|_{H^3}^2  + \| \eta \|_{H^3}^3 \right].
\eeq

We claim that:
\beq  \label{eq:etaH3bound}
  \norm{\eta(t)}_{H^3} \leq C\,\ep \, \norm{\varphi(t)}^3_{H^3}.
\eeq
Indeed, a lengthy yet elementary calculation, together with the 2D Sobolev embedding theorem, yields
\beq \label{eq:D3etaest}
\|D^3\eta\|_{L^2} \leq C\ep(\|\varphi\|_{H^3}^3 +  \|\varphi\|_{H^3}).\eeq
In addition, it is immediate that
\beq \label{eq:etaL2est}
\|\eta\|_{L^2} \leq 2\ep \|\varphi\|_{L^2}.
\eeq
However, $\|\varphi\|_{H^3} \geq \|\varphi\|_{L^2}=\|\varphi_0\|_{L^2}$, and hence it follows that
\beq \label{eq:lin2cubeest}
\|\varphi\|_{H^3} \leq C \|\varphi\|_{H^3}^3,
\eeq
where $C=\|\varphi_0\|_{L^2}^{-2}$. Putting together \eqref{eq:D3etaest}, \eqref{eq:etaL2est} and \eqref{eq:lin2cubeest} easily yields \eqref{eq:etaH3bound}.

We estimate the powers of $\|\eta\|_{H^3}$ in \eqref{eq:dvH3dt2} using first \eqref{eq:etaH3bound}, then iteratively \eqref{eq:lin2cubeest}, and, finally, recalling $\varepsilon \leq 1$ by assumption \eqref{eq:firstep0cond}, we deduce that
\beq \label{eq:vH3Boundcombined}
\frac{d}{dt} \left\|\bv \right\|_{H^{3}} \leq C \left\|\nabla \bv\right\|_{L^{\infty}} \left[ 1 + \| \bv \|_{H^3} + \varepsilon\| \varphi \|_{H^3}^9\right].
\eeq

Let us now set
\beq \label{eq:YZnotation}
    Y(t):=  \norm{\bv(t,\cdot)}_{H^3}, \quad Z(t):= \norm{\varphi(t,\cdot)}_{H^3}.
\eeq

We make use of  estimate \cite[(2.7)]{BLS20} on $\norm{\varphi}_{H^3}$ to find:
\beq
   \frac{d}{dt} \norm{\varphi}_{H^3}
    \leq C\,\left( \norm{\bv}_{H^3} \norm{\nabla \varphi}_{L^\infty} +\|\nabla \bv\|_{L^\infty} \,\norm{\varphi}_{H^3} \right).  \label{eq:nhEulerEnergyEst.2}
\eeq
From the Sobolev inequality it follows that
\[ \frac{d}{dt} \norm{\varphi}_{H^3}  \leq C\,  \norm{\bv}_{H^3} \norm{\varphi}_{H^3},
\]
so that, using  Gr\"onwall's Lemma, we obtain
\begin{align} \label{eq:varphiH3est}
Z(t)& \leq \left\|\varphi_{0}\right\|_{H^{3}}\exp\left(C\int^{t}_{0}Y(\tau)d\tau\right)\\
& \leq C \exp\left(C\int^{t}_{0}Y(\tau)d\tau\right) . \nonumber
\end{align}

Additionally, we easily find that
\beq \label{eq:varphiH3estpower9}
[Z(t)]^9 \leq C \exp\left(C\int^{t}_{0}Y(\tau)d\tau\right) ,
\eeq
adjusting the constants as needed.

Thus, using the notation introduced in \eqref{eq:YZnotation} together with estimate \eqref{eq:varphiH3estpower9}, we see that \eqref{eq:vH3Boundcombined} becomes
\beq \label{eq:dvH3dt3}
\frac{dY}{dt}  \leq C \left\|\nabla \bv\right\|_{L^{\infty}} \left[ 1 + Y + \varepsilon \exp\left(C\int^{t}_{0}Y(\tau)d\tau\right)\right].
\eeq

We will now use an improvement of the potential theory estimate \eqref{eq:BoussinesqEnergyEst.2}. This version is due to H. Kozono and Y. Taniuchi, see \cite[Theorem 1]{KT2000}, using additionally that $\| D\bv \|_{BMO} \leq C \|\upomega \|_{L^\infty}$ and $\log^+(x) \leq \log (1+x)$, namely
\begin{equation} \label{eq:K-TverBoussinesqEnergyEst.2}
\norm{D \bv}_{L^\infty} \leq C\, \left[1+\left(\log(1+ \norm{\bv}_{H^{3}}) \right) \norm{\upomega}_{L^\infty}\right]. \end{equation}

We write \eqref{eq:K-TverBoussinesqEnergyEst.2} using the notation $Y= \| \bv \|_{H^3}$ as:
\[\| \nabla \bv \|_{L^\infty} \leq C\left[ 1+\left( \log \left(1+Y\right)\right) \| \upomega \|_{L^{\infty}}\right] .\]

Inserting this estimate into \eqref{eq:dvH3dt3} gives
\beq \label{eq:vH3normEnergyEstNEWN}
\frac{dY}{dt}  \leq C \left[ 1+\left( \log \left(1+Y\right)\right) \| \upomega \|_{L^{\infty}}\right]\left[ 1 + Y + \varepsilon \exp\left(C\int^{t}_{0}Y(\tau)d\tau\right)\right].
\eeq

Next we recall the vorticity equation \eqref{eq:vorteqetaform}, from which we deduce that
\[\left\|\upomega \right\|_{L^{\infty}} \leq \left\|\upomega_{0} \right\|_{L^{\infty}}+\int^{t}_{0}\left\|\nabla \eta\right\|_{L^{\infty}} \left\|\nabla p\right\|_{L^{\infty}}ds.\]

Using \eqref{eq:pH2Bound} and \eqref{eq:etaH3bound} and expressing $\|\bv \|_{H^3}$ and $\| \varphi \|_{H^3}$ in terms of $Y$ and $Z$, respectively, we obtain:
\[\left\|\upomega \right\|_{L^{\infty}} \leq C + C\ep \int^{t}_{0}Z^3 Y(1+Y+\ep^2 Z^{6}).\]

Hence \eqref{eq:vH3normEnergyEstNEWN} becomes
\beq \label{eq:Ydiffineq1stver}
\begin{split}
\frac{dY}{dt} \leq &C  \left\{1+\bigg[\log (1+Y)\bigg] \left[ 1 + \ep\int_0^t Z^3 Y(1+Y+\ep^2 Z^{6}) \,ds\right]\right\}\\
&\times \left\{1+Y+\ep \exp\left[C\int^{t}_{0}Y(\tau)d\tau\right]\right\}.
\end{split}
\eeq


We  write \eqref{eq:Ydiffineq1stver} in compact form as
\beq \label{eq:Ydiffineq2ndver}
\frac{dY}{dt} \leq C  \left\{(1+Y)\big[1+\log(1+Y)\big] + \ep F(t,Y,Z)\right\},
\eeq
where $F$ is a nonlinearity involving $(Y(t), Z(t))$ and their histories up until time $t$. More precisely,
\[
\begin{split}
F(t, Y,Z)&= \bigg[1+\log(1+Y)\bigg]\left\{\exp \left[C\int^{t}_{0}Y(\tau)d\tau\right]\right\} \\
& + (1+Y)\bigg[\log(1+Y)\bigg]
\left\{ \int_0^t Z^3 Y (1+Y+\ep^2 Z^6) \right\} \\
& + \ep \bigg[\log(1+Y)\bigg] \left\{ \int_0^t Z^3 Y (1+Y+\ep^2 Z^6) \right\} \left\{\exp \left[C\int^{t}_{0}Y(\tau)d\tau\right]\right\}.
\end{split}
\]

Now, since $Y\geq 0$, $Z \geq 0$ it is easily verified that $F(t,Y,Z) \geq 0$. Therefore $\ep F(t,Y,Z) \leq \ep (1+Y)F(t,Y,Z).$ Hence, \eqref{eq:Ydiffineq2ndver} implies

\beq \label{eq:Ydiffineq2ndverprime}
\frac{dY}{dt} \leq C  (1+Y)\left\{\big[1+\log(1+Y)\big] + \ep F(t,Y,Z)\right\}.
\eeq

Let $\Upsilon:= 1+Y$. Then
\beq \label{eq:Ydiffineq3rdver}
\frac{d\Upsilon}{dt} \leq C \Upsilon [ (1+\log\Upsilon  + \ep F(t,\Upsilon-1,Z)].
\eeq
From \eqref{eq:varphiH3est} it follows that
\beq
Z(s) \leq C \exp\left \{C \int_0^s (\Upsilon (\tau)-1) \, d\tau\right\}, \qquad \text{ for all } 0\leq s \leq t.
\eeq
We note that $F$ is {\it monotone non-decreasing} with respect to $Y$ and $Z$ and their histories, that is, if $Y_1(s) \leq Y_2(s)$ and $Z_1(s) \leq Z_2(s)$ for all $0\leq s \leq t$, then $F(t,Y_1,Z_1) \leq F(t,Y_2,Z_2)$. In view of this observation we deduce that $\Upsilon$ satisfies the following inequality
\beq \label{eq:yIneqNhEulerFinal}
\frac{d\Upsilon}{dt} \leq C \Upsilon [ (1+\log\Upsilon  + \ep \mathcal{H}(t,\Upsilon)],
\eeq
with $\mathcal{H}$ a nonlinear function of $\Upsilon$, $t$ and time-integrals of nonlinear functions of $\Upsilon$ and its time-integral:
\beq \label{eq:FormulaHtUpsilon}
\mathcal{H}(t,\Upsilon) = F\left(t,\Upsilon-1, \exp\left \{C \int_0^t (\Upsilon (s)-1) \, ds\right\} \right).
\eeq
We note that $\mathcal{H}(t,\Upsilon) \geq 0$ whenever $\Upsilon \geq 1$.

As we did for the Boussinesq equations, we first study the equation
\beq \label{eq:mathcalZODEnHEuler}
\frac{d\mathcal{Z}}{dt} = C \mathcal{Z} [ (1+\log \mathcal{Z}  + \ep \mathcal{H}(t,\mathcal{Z})],
\eeq
$\mathcal{Z}(0)=\mathcal{Z}_0>1$. We point out that, in view of the sign properties of $\mathcal{H}$, inherited from $F$, it follows that $\mathcal{Z}(t) \geq 1$ for all $t$ in the interval of existence for \eqref{eq:mathcalZODEnHEuler}.

We make the same change of variables $X= 1+\log \mathcal{Z}$ followed by $W= e^{-Ct}X$ to obtain
\beq \label{eq:WODEnHEuler}
\frac{dW}{dt} = C \ep  e^{-Ct} \mathcal{H}\left(t,\exp\{e^{Ct}W\}\right) \equiv \ep \mathcal{F}(t,W).
\eeq
We continue following the strategy for the Boussinesq equations by writing \eqref{eq:WODEnHEuler} in mild form and then using the Banach Contraction Mapping Theorem to establish existence of a solution $W \in C^1([0,T];[0,+\infty))$ to \eqref{eq:WODEnHEuler} for sufficiently small $\ep$. We note, from the expression for $F=F(t,Y,Z)$, combined with the changes of variables performed, that $\mathcal{F}$ inherits properties \eqref{i:CP1} and \eqref{i:CP2} from similar properties satisfied by $F$. This allows us to establish an analogue of Proposition \ref{prop:BoussinesqODEexist} for \eqref{eq:mathcalZODEnHEuler}.

It was noted in \eqref{eq:Mcontnorm} that, while
\[\mathcal{M}(t)\equiv \|\varphi(t,\cdot)\|_{H^3}+\|\eta(t,\cdot)\|_{H^3}+\|\bv(t,\cdot)\|_{H^3}\]
remains bounded, the time of existence may be prolonged. However, we showed in \eqref{eq:etaH3bound} that $\|\eta(t,\cdot)\|_{H^3}$ is controlled by $\|\varphi(t,\cdot)\|_{H^3}$ and, in \eqref{eq:varphiH3est}, that
$\|\varphi(t,\cdot)\|_{H^3}$ is controlled by $\|\bv(t,\cdot)\|_{H^3}$. Therefore we may use
$\|\bv(t,\cdot)\|_{H^3}$ as a continuation norm, in place of $\mathcal{M}(t)$.

Putting back the dependence on $\ep$ in $\bv^\ep$, $\varphi^\ep$ and $\eta^\ep$, we now argue that the proof of Theorem \ref{t:MainNhEuler} proceeds similarly to the proof of Theorem \ref{t:MainBoussinesq} using $\Upsilon^\ep \equiv \|\bv^\ep\|_{H^3} +1$ in place of $y^\ep$. Indeed, the argument relies on the ODE for $\mathcal{Z}$, \eqref{eq:mathcalZODEnHEuler}, preserving the condition $\mathcal{Z}\geq 1$, as has already been observed, and on $\mathcal{H}$ being monotone non-decreasing with respect to $\Upsilon$ and its history, something which can be easily verified.

Again for fixed initial data, we have that
\[
   \ep \sim e^{-e^{e^{CT}}},
\]
since $\mathcal{F}$ depends exponentially on $W$ through $\mathcal{H}$ and $\mathcal{H}$ depends exponentially
on $\mathcal{Z}$, which depends exponentially on $e^{Ct} \, W$.

\vspace{1cm}

\begin{remark}
 We observe that both inequalities \eqref{eq:yIneqBoussinesq} and \eqref{eq:yIneqNhEulerFinal}
 are of the general form:
\beq \label{eq:yIneqAbstract}
  \dt y \leq G(y)[1 + \ep \Bar{F}(t,y)],
\eeq
where $G$ satisfies Osgood's condition for global existence and uniqueness of solutions of the  autonomous ODE \, $\dt{y}=G(y)$, $\Bar{F}$ is non linear and history dependent, but locally Lipschitz in $y$, continuous in $t$, and monotone non-decreasing.
Under these conditions, it is possible to prove a general long-time existence result for solutions to \eqref{eq:yIneqAbstract} for $\ep$ sufficiently small. However, we specialize our proofs to \eqref{eq:yIneqBoussinesq} and \eqref{eq:yIneqNhEulerFinal} in order to estimate the dependence of $\ep$ on $T$ and the size of the initial condition or, equivalently, the dependence of $T$ on $\ep$ and the size of the $y_0$ (cf. Remark \ref{r:EpsilonGrowth}).
\end{remark}

\bibliography{NonhomEulerBouss-2025_04_29-Hantaek-L2-mixVersion}

\begin{thebibliography}{10}

\bibitem{BLS20}
H.~Bae, W.~Lee, and J.~Shin.
\newblock A blow-up criterion for the inhomogeneous incompressible {E}uler
  equations.
\newblock {\em Nonlinear Anal.}, 196:111774, 9, 2020.

\bibitem{BKM1984}
J.~T. Beale, T.~Kato, and A.~Majda.
\newblock Remarks on the breakdown of smooth solutions for the 3-{D} {Euler}
  equations.
\newblock {\em Commun. Math. Phys.}, 94:61--66, 1984.

\bibitem{BV80.1}
H.~Beir\~{a}o~da Veiga and A.~Valli.
\newblock On the {E}uler equations for nonhomogeneous fluids. {I}.
\newblock {\em Rend. Sem. Mat. Univ. Padova}, 63:151--168, 1980.

\bibitem{BV80.2}
H.~Beir\~{a}o~da Veiga and A.~Valli.
\newblock On the {E}uler equations for nonhomogeneous fluids. {II}.
\newblock {\em J. Math. Anal. Appl.}, 73(2):338--350, 1980.

\bibitem{BrezisMironescu2018}
H.~Brezis and P.~Mironescu.
\newblock Gagliardo-{N}irenberg inequalities and non-inequalities: the full
  story.
\newblock {\em Ann. Inst. H. Poincar\'{e} C Anal. Non Lin\'{e}aire},
  35(5):1355--1376, 2018.

\bibitem{CalderonZygmund1952}
A.~Calderon and A.~Zygmund.
\newblock On the existence of certain singular integrals.
\newblock {\em Acta Mathematica}, 88(1):85 – 139, 1952.
\newblock Cited by: 683.

\bibitem{CL03}
D.~Chae and J.~Lee.
\newblock Local existence and blow-up criterion of the inhomogeneous {E}uler
  equations.
\newblock {\em J. Math. Fluid Mech.}, 5(2):144--165, 2003.

\bibitem{CN97}
D.~Chae and H.-S. Nam.
\newblock Local existence and blow-up criterion for the {B}oussinesq equations.
\newblock {\em Proc. Roy. Soc. Edinburgh Sect. A}, 127(5):935--946, 1997.

\bibitem{D10}
R.~Danchin.
\newblock On the well-posedness of the incompressible density-dependent {E}uler
  equations in the {$L^p$} framework.
\newblock {\em J. Differential Equations}, 248(8):2130--2170, 2010.

\bibitem{D13}
R.~Danchin.
\newblock Remarks on the lifespan of the solutions to some models of
  incompressible fluid mechanics.
\newblock {\em Proceedings of the American Mathematical Society},
  141(6):1979--1993, 2013.

\bibitem{DF11}
R.~Danchin and F.~Fanelli.
\newblock The well-posedness issue for the density-dependent {E}uler equations
  in endpoint {B}esov spaces.
\newblock {\em J. Math. Pures Appl. (9)}, 96(3):253--278, 2011.

\bibitem{FZ21}
J.~Fan and Y.~Zhou.
\newblock A blow-up criterion of the ideal density-dependent flows.
\newblock {\em J. Math. Anal. Appl.}, 497(1):Paper No. 124881, 5, 2021.

\bibitem{FM24}
L.~C.~F. Ferreira and D.~F. Machado.
\newblock On the well-posedness in {B}esov-{H}erz spaces for the inhomogeneous
  incompressible {E}uler equations.
\newblock {\em Dyn. Partial Differ. Equ.}, 21(1):1--29, 2024.

\bibitem{Gagliardo1959}
E.~Gagliardo.
\newblock Ulteriori propriet\`a di alcune classi di funzioni in pi\`u
  variabili.
\newblock {\em Ricerche Mat.}, 8:24--51, 1959.

\bibitem{I94}
S.~Itoh.
\newblock Cauchy problem for the {E}uler equations of a nonhomogeneous ideal
  incompressible fluid.
\newblock {\em J. Korean Math. Soc.}, 31(3):367--373, 1994.

\bibitem{Kato1986}
T.~Kato.
\newblock Remarks on the {E}uler and {N}avier-{S}tokes equations in {${\bf
  R}^2$}.
\newblock In {\em Nonlinear functional analysis and its applications, {P}art 2
  ({B}erkeley, {C}alif., 1983)}, volume 45, Part 2 of {\em Proc. Sympos. Pure
  Math.}, pages 1--7. Amer. Math. Soc., Providence, RI, 1986.

\bibitem{KT2000}
H.~Kozono and Y.~Taniuchi.
\newblock Limiting case of the sobolev inequality in {BMO}, with application to
  the euler equations.
\newblock {\em Communications in Mathematical Physics}, 214(1):191--200, 2000.

\bibitem{MajdaBertozzi}
A.~J. Majda and A.~L. Bertozzi.
\newblock {\em Vorticity and incompressible flow}, volume~27 of {\em Cambridge
  Texts in Applied Mathematics}.
\newblock Cambridge University Press, Cambridge, 2002.

\bibitem{MP19}
U.~Manna and A.~A. Panda.
\newblock Higher order regularity and blow-up criterion for semi-dissipative
  and ideal {B}oussinesq equations.
\newblock {\em J. Math. Phys.}, 60(4):041503, 22, 2019.

\bibitem{Nirenberg1959}
L.~Nirenberg.
\newblock On elliptic partial differential equations.
\newblock {\em Ann. Scuola Norm. Sup. Pisa Cl. Sci. (3)}, 13:115--162, 1959.

\bibitem{W13}
Z.~Wei.
\newblock Local well-posedness for density-dependent incompressible {E}uler
  equations.
\newblock {\em Electron. J. Differential Equations}, pages No. 146, 18, 2013.

\bibitem{Z10}
Y.~Zhou.
\newblock Local well-posedness and regularity criterion for the density
  dependent incompressible {E}uler equations.
\newblock {\em Nonlinear Anal.}, 73(3):750--766, 2010.

\end{thebibliography}
\bibliographystyle{abbrv}

\end{document}